\documentclass{amsart}

\usepackage{amsmath,amsthm, amssymb,mathrsfs,  xypic}
\usepackage{stackengine,scalerel}
\usepackage[colorlinks]{hyperref}
\usepackage[capitalize]{cleveref}

\newcommand{\Hom}{\mathrm{Hom}}

\newcommand{\proet}{\mathrm{pro\acute{e}t}}
\newcommand{\Lie}{\mathrm{Lie}}

\newcommand{\Ext}{\mathrm{Ext}}

\newcommand{\GL}{\mathrm{GL}}

\newcommand{\et}{\mathrm{\acute{e}t}}

\newcommand{\crys}{\mathrm{crys}}

\newcommand{\SL}{\mathrm{SL}}

\newcommand{\mr}[1]{\mathrm{#1}}
\newcommand{\ul}[1]{\underline{#1}}
\newcommand{\mbb}[1]{\mathbb{#1}}

\theoremstyle{plain}

\newtheorem*{theorem*}{Theorem}
\newtheorem*{conjecture*}{Conjecture}
\newtheorem{theorem}{Theorem}

\newtheorem{lemma}{Lemma}

\theoremstyle{definition}

\newtheorem{remark}{Remark}
\newtheorem{question}{Question}

\crefformat{section}{\S#2#1#3}
\crefformat{Section}{\S#2#1#3}

\title[The de Jong fundamental groups of a non-triv. ab. var. is non-abelian]{The de Jong fundamental group of a non-trivial abelian variety is non-abelian}

\author{Sean Howe}
\date{\today}

\begin{document}

\begin{abstract}
We show that the de Jong fundamental group of any non-trivial abelian variety over a complete algebraically closed extension $C/\mathbb{Q}_p$ is non-abelian. Generalizing an argument for $\mathbb{P}^1_C$, we also show that the de Jong fundamental group of any connected rigid analytic variety over $C$ admitting a non-constant map to $\mathbb{P}^n_C$ (e.g., an abelian variety) depends on $C$ and is big. 
\end{abstract}

%% Basic setup commands
% If you don't want a title page comment out the next line and uncomment the line after it:
\maketitle
%\omittitle

%% Make the various tables of contents

\section{Introduction}
Let $C$ be an algebraically closed field of characteristic zero. If $A/C$ is an abelian variety and $0 \in A(C)$ is the identity element, then the \'{e}tale fundamental group $\pi_{1,\et}(A,0)$ is identified with $T_{\hat{\mathbb{Z}}}A=\varprojlim_n A[n](C)$ by its translation action on the universal profinite \'{e}tale cover $\tilde{A}=\varprojlim_n A$ (in both limits the indices are positive integers ordered by divisibility, and the transition map for $m|n$ is multiplication by $n/m$). If $C$ is a complete algebraically closed extension of $\mathbb{Q}_p$, then we can also pass to the associated rigid analytic space $A^{\mathrm{rig}}/C$. Its category of finite \'{e}tale covers is equivalent to the category of finite \'{e}tale covers of $A$, so that \[ \pi_{1,\et}(A^\mathrm{rig},0)=\pi_{1,\et}(A,0)=T_{\hat{\mathbb{Z}}}A.\]
In particular, $\pi_{1,\et}(A^{\mr{rig}},0)$ is an abelian group. 

However, in the rigid analytic setting, there are more general notions of covering spaces that mix properties of topological and finite \'{e}tale coverings. One such notion is that of de Jong coverings as introduced in \cite{deJong.EtaleFundamentalGroupsOfNonArchimedeanAnalyticSpaces}, and there is an associated de Jong fundamental group $\pi_{1, \mathrm{dJ}}(A^{\mr{rig}},0)$, a pro-discrete topological group whose continuous actions on discrete sets correspond to de Jong coverings of $A^{\mr{rig}}$. It is well known that de Jong fundamental groups can be much larger than classical \'{e}tale fundamental groups: for example, although $\pi_{1,\et}(\mathbb{P}^{n,\mr{rig}}_C)=\pi_{1,\et}(\mathbb{P}^{n}_C)=\{1\}$, for $n \geq 1$ the de Jong fundamental group $\pi_{1,\mathrm{dJ}}(\mathbb{P}^{n,\mr{rig}}_C)$ is large: a connected component over $C$ of the height $n+1$ Lubin-Tate tower furnishes, via the Gross-Hopkins period map, a tower of connected de Jong coverings of $\mathbb{P}^{n,\mr{rig}}_C$. For $y \in \mathbb{P}^{n,\mr{rig}}(C)$, this tower is classified by a continuous homomorphism
\begin{equation}\label{eq.deJong-surjections} \rho_n: \pi_{1,\mathrm{dJ}}(\mathbb{P}^{n,\mr{rig}}_C, y) \rightarrow \SL_{n+1}(\mathbb{Q}_p), \end{equation}
and this map is a surjection by \cite[Proposition 7.4]{deJong.EtaleFundamentalGroupsOfNonArchimedeanAnalyticSpaces}. 

For abelian varieties over $C$ of bad reduction, there are well-known connected de Jong coverings with infinite discrete fibers coming from  Raynaud uniformization. These are all abelian and, moreover, the associated covering groups are free $\mathbb{Z}$-modules of finite rank so that, up to profinite completion, these covers are detected already by the classical \'{e}tale fundamental group. The main point of this note is that there are abundant de Jong covers beyond this construction. In particular:
\begin{theorem}\label{theorem.non-abelian}
    Let $C/\mathbb{Q}_p$ be a complete algebraically closed extension, let $A/C$ be an abelian variety, and let $0 \in A(C)$ be the identity element. If the dimension of $A$ is at least $1$, then $\pi_{1,\mathrm{dJ}}(A^{\mr{rig}}, 0)$ is not an abelian group.  
\end{theorem}

In particular, in the setting of \cref{theorem.non-abelian}, the natural map $\pi_{1, \mathrm{dJ}}(A^{\mathrm{rig}},0)\rightarrow \pi_{1, \et}(A^\mathrm{rig},0)$ is not injective and thus finite \'{e}tale coverings do not detect all de Jong coverings. This is in contrast to the complex analytic setting where the topological fundamental group of a complex abelian variety is a finite free $\mathbb{Z}$-module and thus injects into its profinite completion, the \'{e}tale fundamental group. 

\begin{remark}\label{remark.surprising} From a classical perspective, an abelian variety having a non-abelian fundamental group is surprising. On the other hand, it is perhaps not as surprising as the non-triviality  of the fundamental group of $\mathbb{P}^{n,\mr{rig}}_C$ furnished already by the covers of \cref{eq.deJong-surjections}, and we use these covers in our proof of \cref{theorem.non-abelian}. Note that, in \cite{Howe.ThedeJongFundamentalGroupOfP1C}, we used these covers for $n=1$ to give two constructions of many non-isomorphic local systems on $\mathbb{P}^{1,\mr{rig}}_C$, then deduced that the fundamental group of $\mathbb{P}^{1,\mr{rig}}_C$ depends on the choice of $C$ and can be big. In \S\ref{s.big}, we generalize both constructions to obtain the same result for the fundamental group of any connected rigid analytic variety over $C$ admitting a non-constant map to $\mathbb{P}^{n,\mr{rig}}_C$ (e.g., an abelian variety).
\end{remark}

For $n=\dim A$, we establish \cref{theorem.non-abelian} by pulling back the cover of $\mathbb{P}^{n,\mr{rig}}_C$ in \cref{eq.deJong-surjections} along a generically \'{e}tale map $A \rightarrow \mathbb{P}^n_C$ then using geometric Sen theory to show the monodromy of this pullback remains non-abelian. The proof also implies that the monodromy is infinite, an observation we first learned from Sasha Petrov.

\begin{question} This method of proof does not work for a very general abeloid variety $A/C$ --- indeed, such an $A$ may not admit \emph{any} rational functions, in which case the covers of \cref{eq.deJong-surjections} cannot be used to produce any candidate non-abelian covers of $A$. It is thus natural to ask: is the de Jong fundamental group of a non-trivial abeloid non-abelian? We note that, for $A/C$ abeloid, the usual profinite \'{e}tale fundamental group is ``as expected," i.e. equal to $T_{\hat{\mathbb{Z}}}A$ (see \cite[Appendix A]{BellovinCaiHowe.CharacterizingPerfectoidCoversOfAbelianVarieties}). 
\end{question}

The key ingredients in the proof of \cref{theorem.non-abelian} are the injectivity of the geometric Sen morphism for the Lubin-Tate tower (\cref{lemma.sen-comp}, which follows from the computation of the geometric Sen morphism given in \cite{DospinescuRodriguezCamargo.JacquetLanglands} or a more general computation in \cite{Howe.GeometricSenAndKodairaSpencer}), and the observation that, if the monodromy is abelian, then the geometric Sen morphism must be constant (\cref{lemma.constant-Sen-abelian-torsor}). We treat these preliminaries in \cref{s.preliminaries} and then use them to prove \cref{theorem.non-abelian} in \cref{s.proof}. In \S\ref{s.big} we state and prove Theorem \ref{theorem.generators}, the generalization of \cite[Theorem 1] {Howe.ThedeJongFundamentalGroupOfP1C} that was described in \cref{remark.surprising}.  

\begin{remark}
Our proof of Theorem \ref{theorem.non-abelian} does not use the surjectivity of \cref{eq.deJong-surjections}; all that is needed is the pointwise injectivity of the associated geometric Sen morphism (a ``softer" fact than the actual monodromy computation for \cref{eq.deJong-surjections}). 
\end{remark}

\subsection*{Acknowledgements} We thank Rebecca Bellovin, Hanlin Cai, and Ben Heuer for helpful discussions about profinite \'{e}tale covers of abelian varieties around \cite{BellovinCaiHowe.CharacterizingPerfectoidCoversOfAbelianVarieties}. We thank Sasha Petrov for helpful comments on an earlier draft and for pointing out to us that the geometric Sen morphism can be used to show that, for $X/C$ a connected rigid  variety and $f: X \rightarrow \mathbb{P}^n_C$ a map that is not locally constant, the pullback of the Lubin-Tate local system along $f$ is non-trivial. This work was supported in part by National Science Foundation grants DMS-2201112 and DMS-2501816.

\section{Preliminaries}\label{s.preliminaries}

In this section we establish the technical results needed in the proof of \cref{theorem.non-abelian}.  

\subsection{The geometric Sen morphism and its functoriality}
Recall from \cite[Theorem 1.0.4] {RodriguezCamargo.GeometricSenTheoryOverRigidAnalyticSpaces} that, for $X/C$ a smooth rigid analytic variety and $K$ a $p$-adic Lie group with Lie algebra $\mathfrak{k}$, any pro-\'{e}tale $K$-torsor $\tilde{X}/X$ gives rise to a geometric Sen morphism
\[ \theta_{\tilde{X}} \in \left(\Omega^1_{X/C}\otimes_{\mathcal{O}} \left(\tilde{\mathfrak{k}}  \otimes_{\ul{\mathbb{Q}_p}} {\widehat{\mathcal{O}}}(-1)\right) \right)(X), \]
where $\tilde{\mathfrak{k}}=\tilde{X} \times^{K} \ul{\mathfrak{k}}$ is the twisted form of the constant local system $\ul{\mathfrak{k}}$ associated to $\tilde{X}$ and the adjoint action of $K$ on $\mathfrak{k}$. Dually, we view this as a morphism on $X_\proet$,
\[ \kappa_{\tilde{X}}: T_{X/C} \otimes_{\mathcal{O}} \hat{\mathcal{O}} \rightarrow \tilde{\mathfrak{k}}\otimes_{\ul{\mbb{Q}_p}} \hat{\mathcal{O}}(-1). \]

We recall also from the statement of \cite[Theorem 1.0.4] {RodriguezCamargo.GeometricSenTheoryOverRigidAnalyticSpaces} that $\kappa_{\tilde{X}}$ is natural for pullbacks and push-outs of torsors:
\begin{lemma}[Functoriality of $\kappa$] \label{lemma.functoriality} Let $K$ be a $p$-adic Lie group. 
\begin{enumerate}
    \item If $f:X \rightarrow Y$ is a map of smooth rigid analytic varieties over $C$ and $\tilde{Y}/Y$ is a pro-\'{e}tale $K$-torsor, then $\kappa_{f^*\tilde{Y}}=f^*\kappa_{\tilde{X}} \circ df$, where $df: T_{X/C} \rightarrow \pi^*T_{Y/C}$ 
    is the derivative of $f$.
    \item If $H$ is a $p$-adic Lie group  and $\rho: H\rightarrow K$ is a continuous homomorphism, then for any smooth rigid analytic variety $X/C$ and pro-\'{e}tale $H$-torsor $\tilde{X}/X$, 
 $\kappa_{\tilde{X} \times^{\rho} \ul{H}} = d\rho \circ \kappa_{\tilde{X}}$, where $\mathfrak{h}:=\Lie H$ and $d\rho: \mathfrak{h} \rightarrow \mathfrak{k}$
    is the derivative.
\end{enumerate}
\end{lemma}

\subsection{Representations of the de Jong fundamental group} Suppose $X/C$ is a smooth connected rigid analytic variety and $x \in X(C)$. If $K$ is a $p$-adic Lie group and $\rho: \pi_{1,\mathrm{dJ}}(X,x) \rightarrow K$ is a continous homomorphism, then the discrete $\pi_{1,\mathrm{dJ}}(X,x)$-sets $K/U$, as $U$-varies over open subgroups of $K$, give rise to a tower $\tilde{X}_{\rho}=(X_{K/U, \rho})_U$ of \'{e}tale covers of $X$. Viewed as an object of $X_{\proet}$, $\tilde{X}_{\rho}$ is a $K$-torsor. In this case, we write $\kappa_{\rho}=\kappa_{\tilde{X}_\rho}$ for the associated geometric Sen morphism.  

In particular, we can apply this to the $\rho_n$ of \cref{eq.deJong-surjections}, whose associated torsor is a geometric connected component of the infinite level Lubin-Tate tower. We note that, the associated $\mathbb{Q}_p$-local system $\mathbb{L}$, viewed as a sheaf on the $v$-site $\mathbb{P}^{n,\mr{rig}}_{C,v}$ (or on perfectoid objects in $\mathbb{P}^{n,\mr{rig}}_{C,\proet}$), admits a construction by modifications of vector bundles on the Fargues-Fontaine curve as in \cite[Lecture 24]{ScholzeWeinstein.BerkeleyLecturesOnPAdicGeometryAMS207}. In particular, it sits in a short exact sequence involving the period sheaf $\mathbb{B}$ sending a perfectoid $S/\mathbb{P}^{n,\mr{rig}}_C$ to $\mathcal{O}(Y_{S^\flat})$ (for $Y_{S^\flat}$ as in \cite[Definition II.1.15]{Fargues.GeometrizationOfTheLocalLanglandsCorrespondenceAnOverview}). Explicitly, for $\varphi$ the Frobenius on $\mathbb{B}$ and $\theta: \mathbb{B} \rightarrow \mathcal{O}$ Fontaine's map, we can write this short exact sequence as
\begin{equation}\label{eq.ses-defining} 0 \rightarrow \mathbb{L} \rightarrow \mathbb{B}^{\varphi^{n+1}=p} \xrightarrow{b \mapsto \sum_{i=1}^{n+1} \theta(\varphi^{i-1}(b)) \cdot z_i}\mathcal{O}(1) \rightarrow 0 \end{equation}
where here $\mathcal{O}(1)$ denotes the usual line bundle on $\mathbb{P}^{n,\mr{rig}}_C$ (viewed as a $v$-bundle) and $z_1, \ldots, z_{n+1}$ are its global sections giving the usual homogeneous coordinates.

\begin{lemma}\label{lemma.sen-comp}
    For $\rho_n$ as in \cref{eq.deJong-surjections}, the geometric Sen morphism $\kappa_{\rho_n}$ is an injection at every geometric point. 
\end{lemma}
\begin{proof}
    By \cref{lemma.functoriality}-(2), it suffices to treat $\rho_n$ as a map to $\GL_{n+1}(\mathbb{Q}_p)$. We note that over $\mathbb{P}^n$ we have a universal filtration on the trivial $\GL_{n+1}$-torsor and thus, via the adjoint representation, a filtration on $\mathfrak{gl}_{n+1} \otimes_{\mathbb{Q}_p} \hat{\mathcal{O}}$. We also have the Hodge-Tate filtration on $ \tilde{\mathfrak{gl}}_{n+1}\otimes_{\ul{\mbb{Q}_p}} \hat{\mathcal{O}}$. As a specific case of the computation given for local Shimura varieties in \cite[\S 4.3]{DospinescuRodriguezCamargo.JacquetLanglands},  $\kappa_{\rho_n}$ is then canonically identified with the map
    \[ T_{\mathbb{P}^n_C} \otimes_{\mathcal{O}} \hat{\mathcal{O}} = \mathrm{gr}^{-1}_{\mathrm{univ}} \left(\mathfrak{gl}_{n+1} \otimes_{\mathbb{Q}_p} \hat{\mathcal{O}}\right) =\mathrm{gr}^1_{\mathrm{HT}}\left( \tilde{\mathfrak{gl}}_{n+1}\otimes_{\ul{\mbb{Q}_p}} \hat{\mathcal{O}}\right)(-1) \]
where the second equality is the Hodge-Tate comparison and we note that, because the filtrations are minuscule, the last term is equal to
\[ \mathrm{Fil}^1_{\mathrm{HT}}\left( \tilde{\mathfrak{gl}}_{n+1}\otimes_{\ul{\mbb{Q}_p}} \hat{\mathcal{O}}\right)(-1) \subseteq \tilde{\mathfrak{gl}}_{n+1}\otimes_{\ul{\mbb{Q}_p}} \hat{\mathcal{O}}(-1).\]
Alternatively, this computation follows from a more general computation relating the geometric Sen and Kodaira-Spencer morphisms for de Rham torsors  \cite{Howe.GeometricSenAndKodairaSpencer}. Indeed, this $\GL_{n+1}(\mathbb{Q}_p)$-torsor is the base change to $C$ of a de Rham $\GL_{n+1}(\mathbb{Q}_p)$ torsor on $\mathbb{P}^{n,\mathrm{rig}}_{\breve{\mathbb{Q}}_p}$ whose associated filtered $\mathrm{GL}_{n+1}$-torsor with integrable connection is, by construction, the trivial torsor with the trivial connection and universal filtration.
\end{proof}

\subsection{Abelian torsors}
 If $H$ is an abelian $p$-adic Lie group with Lie algebra $\mathfrak{h}$, then the adjoint action on $\mathfrak{h}$ is trivial. In particular, if $X/C$ is a smooth rigid analytic variety and $\tilde{X}/X$ is an $H$-torsor, $\tilde{\mathfrak{h}}$ is simply the constant local system $\ul{\mathfrak{h}}$ and 
\begin{equation}\label{eq.constant-abelian} \kappa_{\tilde{X}} \in \Hom(T_{X/C} \otimes_{\mathcal{O}} \hat{\mathcal{O}}, \mathfrak{h} \otimes_{\mbb{Q}_p} \hat{\mathcal{O}}(-1)) = H^0(X, \Omega_{X/C}) \otimes_{\mathbb{Q}_p} \mathfrak{h}(-1). \end{equation}

In particular, when $X$ is an abelian variety, we immediately obtain:
\begin{lemma}\label{lemma.constant-Sen-abelian-torsor}
Suppose $A/C$ is an abelian variety, $H$ is an abelian $p$-adic Lie group, and $\tilde{A}/A$ is a pro-\'{e}tale $H$-torsor. Then, the geometric Sen morphism 
\[ \kappa_{\tilde{A}}: \Lie A \otimes_{C} \hat{\mathcal{O}} = T_{A/C} \otimes_{\mathcal{O}} \hat{\mathcal{O}} \rightarrow \mathfrak{h} \otimes_{\mbb{Q}_p} \hat{\mathcal{O}}(-1) \]
is constant, i.e., it is induced by the base change from $C$ to $\hat{\mathcal{O}}$ of a $C$-linear map 
\[ \Lie A \rightarrow \mathfrak{h} \otimes_{\mbb{Q}_p} C(-1). \]
\end{lemma}

\begin{remark}
    The simplification in the case of abelian covers as in \cref{eq.constant-abelian} was observed already in \cite{BellovinCaiHowe.CharacterizingPerfectoidCoversOfAbelianVarieties} where, in particular, we computed the geometric Sen morphism for all profinite \'{e}tale torsors over abelian varieties (which are necessarily abelian). 
\end{remark}

\section{Proof of \cref{theorem.non-abelian}}\label{s.proof}

We now prove \cref{theorem.non-abelian}. 

\begin{proof}[Proof of \cref{theorem.non-abelian}]
    Let $n=\dim A$. We may choose an embedding of $A$ into $\mathbb{P}^m_C$, $m > n$, and then project onto a generic $n$-plane $\mathbb{P}^n_C \subseteq \mathbb{P}^m_C$ to obtain a generically \'{e}tale map $f: A \rightarrow \mathbb{P}^n_C$.  We claim that the pullback of the $\SL_{n+1}(\mathbb{Q}_p)$-torsor associated to $f$ has non-abelian monodromy, i.e., $\rho_n \circ f_*$ has non-abelian image, where $f_*: \pi_{1,\mathrm{dJ}}(A^{\mr{rig}},0) \rightarrow \pi_{1,\mathrm{dJ}}(\mathbb{P}^{n, \mr{rig}}_C, \pi(0))$ is the map induced by $f$. 
    
    To see this, first observe that the map $f$ is proper as $A/C$ is proper and $\mathbb{P}^n_C/C$ is separated. Thus $f$ cannot be \'{e}tale since $\mathbb{P}^n_C$ has no non-trivial connected finite \'{e}tale covers. Thus $df$ is not an isomorphism, so there exists a nonzero vector $v \in \Lie A=H^0(A, T_{A/C})$ and a point $x \in A(C)$ such that $df_x(v)=0$, where $df_x$ is the derivative at that point, a map of finite dimensional $C$-vector spaces $T_{A/C,x} \rightarrow T_{\mathbb{P}^n_C/C, f(x)}$. Since $f$ is generically \'{e}tale, there also exists a point $y \in A(C)$ such that $df_y(v) \neq 0$. 

   Now, suppose $\rho_n \circ \pi_*$ has abelian image, and write $H \leq \SL_n(\mathbb{Q}_p)$ for the (abelian) closure of its image, a $p$-adic Lie group. Write $\mathfrak{h}:=\Lie H$, and write $\rho$ for the induced map $\pi_{1,\mathrm{dJ}}(A^{\mr{rig}},0) \rightarrow H$. Then, by \cref{lemma.constant-Sen-abelian-torsor}, $\kappa_\rho$ is a constant map. Thus it cannot annihilate the vector $v \in \Lie A$ at one point but not another.    However, for $\iota: H \rightarrow \SL_{n+1}(\mathbb{Q}_p)$ the inclusion, the functorialities of \cref{lemma.functoriality} give
    \[ d\iota\circ \kappa_\rho = \kappa_{\rho_n \circ \pi_*}=\pi^*\kappa_{\rho_n} \circ d\pi \]
     In particular, since $d\iota$ and $\pi^*\kappa_{\rho_n}$ are injective at every geometric point (the latter by \cref{lemma.sen-comp}), the kernels of $\kappa_\rho$ and $d\pi$ agree at every geometric point. Thus we obtain a contradiction since $v$ is in the kernel of $d\pi_x$ but not in the kernel of $d\pi_y$. 
\end{proof}

\section{Fundamental groups of quasi-projective rigid analytic varieties depend on $C$ and can be big}\label{s.big}

We now establish the following generalization of \cite[Theorem 1]{Howe.ThedeJongFundamentalGroupOfP1C} that applies, in particular, to abelian varieties (all of the key ideas appear already in  \cite{Howe.ThedeJongFundamentalGroupOfP1C}).

\newcommand{\dJ}{\mathrm{dJ}}
\begin{theorem}\label{theorem.generators}
   Suppose $X/C$ is a connected rigid analytic variety such that there exists a non-constant map $f: X \rightarrow \mathbb{P}^{n,\mr{rig}}_C$ for some $n\geq 1$.  Then, for $x_0$ in $X(C)$, if $S \subseteq \pi_{1,\dJ}(X,x_0)$ is infinite and $2^S$ has cardinality less than that of $C$, then $S$ is not a set of topological generators for $\pi_{1,\dJ}(X, x_0)$. 
\end{theorem}
\begin{remark}\hfill
\begin{itemize}
    \item   Theorem \ref{theorem.generators} applies, e.g., if $X/C$ admits a locally closed embedding into $\mathbb{P}^n$ (i.e. $X/C$ is quasi-projective). In particular, it applies to  abelian varieties. 
    \item As in \cite[Corollary 1]{Howe.ThedeJongFundamentalGroupOfP1C}, Theorem \ref{theorem.generators} implies that $\pi_1(X,x_0)$ can grow arbitrarily large by base change to larger complete algebraically closed $C'/C$.
\end{itemize}
\end{remark}

\begin{proof}[Proof of Theorem \ref{theorem.generators}]

As in the proof of \cite[Theorem 1]{Howe.ThedeJongFundamentalGroupOfP1C}, it suffices to produce a collection of non-isomorphic local systems indexed by a set of the same cardinality as $C$. We give two such constructions, generalizing the constructions of \cite[Lemma 1]{Howe.ThedeJongFundamentalGroupOfP1C} and \cite[Lemma 2]{Howe.ThedeJongFundamentalGroupOfP1C}. The first construction applies for any $C$ but requires a small restriction on $f$ or $X$; it is based on pulling back Lubin-Tate local systems along ramified maps and computing the geometric Sen morphism as in \cite[Lemma 1]{Howe.ThedeJongFundamentalGroupOfP1C}. The second construction applies whenever $C$ has cardinality greater than $2^\mathbb{N}$ (which suffices for the result since the statement of Theorem \ref{theorem.generators} is otherwise vacuous), and is based on computing extensions of the trivial local system by the Lubin-Tate local system and the infinite-dimensionality of the first cohomology of pullbacks of the Lubin-Tate local system (generalizing a very simple part of the argument in \cite{Hansen.ACounterexample}). 

\hfill\\
\noindent\emph{First construction}.
Suppose given a non-constant map $f:X \rightarrow \mathbb{P}^n_C$, $n \geq 1$. This construction will apply under the further assumption that there exists a smooth point $x \in X(C)$ such that $f$ is a submersion at $x$, i.e \[ df_x: T_{X,x} \twoheadrightarrow T_{ \mathbb{P}^{n,\mr{rig}}_C, f(x)}. \]
Note that this is automatic if $n=1$ (since $df_x$ cannot be everywhere zero as the map is not locally constant). For general $n$, one can typically reduce from the existence of a general non-constant $f$ to the case where $f$ is a submersion at a smooth point $x$ as above by composing with projections: this reduction always works if $X$ is quasi-compact using \cite[Theorem 1]{BhattHansen.TheSixFunctorsForZariskiConstructibleSheavesInRigidGeometry} (which implies either the image does not contain all $C$-points, in which case one can project down to a hyperplane, or that the $df_x$ is generically surjective). More generally, if $X$ is of countable type, then the same argument works by applying \cite[Theorem 1]{BhattHansen.TheSixFunctorsForZariskiConstructibleSheavesInRigidGeometry} on a countable covering\footnote{In general, the situation one needs to avoid to run this reduction argument is the case where $f$ is some type of space-filling map, i.e., a map that is nowhere a submersion but still surjective on $C$-points; such a construction seems plausible without the countability assumption.} 
by quasi-compacts then invoking the Baire Category Theorem on $\mathbb{P}^n(C)$ to deduce that if it is not submersive at any $C$-point then it is not surjective on $C$-points. 

Given $f$ and a point $x\in X(C)$ where $f$ is a submersion, by composing with an automorphism of $\mathbb{P}^n_C$, we may assume $f(x)=[0:0:..:1]=0 \in \mathbb{A}^n(C) \subseteq \mathbb{P}^n(C)$. We may then choose a small ball $B$ in the smooth locus $X^{\mr{sm}}(C)$ containing $x$ such that $f|_B$ is a submersion onto its open image $f(B) \subseteq \mathbb{A}^n(C)$.

Now, for $\mathbb{L}$ the Lubin-Tate $\mathbb{Q}_p$-local system of rank $n+1$ on $\mathbb{P}^{n,\mr{rig}}_C$ (associated to $\rho_n$ as in \cref{eq.deJong-surjections}) and $t=(t_1, t_2,\ldots, t_n)\in C^n$, let $\mathbb{L}_t=\varphi_t^*\mathbb{L}$ where $\varphi_t$ is 
\[ \varphi_t: \mathbb{P}^m_C \rightarrow \mathbb{P}^n_C, [z_1:z_2:\ldots :z_{n+1}] \mapsto [(z_1 - t_1z_{n+1})^2: (z_2 - t_2 z_{n+1})^2: \ldots : z_{n+1}^2]. \]
On $\mathbb{A}^m_C$ with coordinates $x_i=\frac{z_{i}}{z_{n+1}}$, this is the map 
\[ (x_1, \ldots, x_n) \mapsto ((x_1-t_1)^2, \ldots, (x_n - t_n)^2). \]
In particular, on $\mathbb{A}^n_C$ the differential $d\varphi_t$ is identically zero only at the point $t$. 

Letting $\kappa_{\bullet}$ denote the geometric Sen morphism of $\bullet$, Lemma \ref{lemma.functoriality}-(1) implies \[\kappa_{\mathbb{L}_t}=\varphi_t^* \kappa_{\mathbb{L}} \circ d\varphi_{t}.\] Invoking \cref{lemma.sen-comp} (noting $\kappa_{\mathbb{L}}$ here is $\kappa_{\rho_n}$ in the statement), it follows that $t$ is the only point in $\mathbb{A}^m_C$ at which $\kappa_{\mathbb{L}_t}$ is identically zero. Applying Lemma \ref{lemma.functoriality}-(1) again,  \[ \kappa_{f^*\mathbb{L}_t}=f^*\kappa_{\mathbb{L}_t} \circ df.\] Because $f|_B$ is a submersion,  for $t \in f(B)$, the locus of points in $B$ where $\kappa_{f^*\mathbb{L}_t}$ vanishes identically is thus $f^{-1}(t) \cap B$. Since these loci are distinct for different $t \in f(B)$, we conclude that the local systems $f^*\mathbb{L}_t$ are non-isomorphic for all $t \in f(B)$. Since the cardinality of the open $f(B) \subseteq C^m$ is the same as the cardinality of $C$, we conclude.

\hfill\\
\emph{Second construction}
As usual, the map $f$ is determined by the line bundle $\mathcal{L}=f^*\mathcal{O}(1)$ on $X$ and the $n+1$ global sections $s_i=f^* z_i$, $i=1,\ldots, n+1$ of $\mathcal{L}$ which generate $\mathcal{L}$. If the $s_i$ are not linearly independent, then the map factors through a hyperplane $\mathbb{P}^{n-1}_C \subseteq \mathbb{P}^n_C$; thus we may assume the $s_i$ are linearly independent. 

Now, let us assume $C$ has cardinality greater than $2^\mathbb{N}$. Then, for $\mathbb{L}$ the rank $n+1$ dimensional Lubin-Tate local system on $\mathbb{P}^{n,\mr{rig}}_C$, we are going to show the set of isomorphism classes of rank $n+2$ local systems on $\mathbb{P}^{n,\mr{rig}}_C$ arising as extensions of $\mathbb{Q}_p$ by $\mathbb{L}$ has cardinality at least that of $C$. We claim it suffices to show that $\Ext^1(\ul{\mathbb{Q}_p}, f^*\mathbb{L})$ has cardinality at least that of $C$. Indeed, in terms terms of representations of the de Jong fundamental group, the extension structures of a representation $V_2$ by a representation $V_1$ realized in a fixed representation $W$ are given by isomorphism classes of pairs consisting of an inclusion $V_1 \rightarrow W$ and a map $W \rightarrow V_2$ inducing $W/V_1 \xrightarrow{\sim} V_2$. Fixing bases, the cardinality of the collection of extensions structures on a fixed $W$ is thus bounded by that of $M_{\dim W \times \dim V_1}(\mathbb{Q}_p) \times M_{\dim V_2 \times \dim W}(\mathbb{Q}_p)$, which is the same as that of $\mathbb{Q}_p$, which is $2^\mathbb{N}$. Thus, when $C$ has cardinality greater than $2^\mathbb{N}$, the Ext group can only have cardinality at least that of $C$ if the set of isomorphism classes of the local systems arising as extensions also has cardinality at least that of $C$. 

To compute extensions, we note that we may work on the $v$-site $X_v$ (this is just a convenience; we could also work on $X_\proet$ by working with a basis of perfectoid objects). Pulling back \cref{eq.ses-defining}, $f^*\mathbb{L}$ sits in its defining short exact sequence on $X_v$,  
\[ 0 \rightarrow f^*\mathbb{L} \rightarrow \mathbb{B}^{\varphi^{n+1}=p} \xrightarrow{b \mapsto \sum_{i=1}^{n+1} \theta(\varphi^{i-1}(b)) \cdot s_i}f^*\mathcal{L} \rightarrow 0. \]
In particular, for any $(c_1, \ldots, c_{n+1}) \in C^{n+1}$, we may pull back this extension along the map
$\ul{\mathbb{Q}_p} \rightarrow f^*\mathcal{L}$ sending $1$ to $\sum c_i s_i$ to obtain an element of $\Ext^1(\ul{\mathbb{Q}_p}, \mathbb{L})$. The global sections of $\mathbb{B}^{\varphi^{n+1}=p}$ on $C$ over the connected rigid analytic variety $X$ are given by $B^{+,\varphi^{n+1}=p}_\crys$, where $B^{+}_{\crys}$ is the usual Fontaine period ring for $C$, thus this construction induces an embedding
\[ C^{n+1}/B^{+,\varphi^{n+1}=p}_\crys \hookrightarrow \Ext^1(\ul{\mathbb{Q}_p}, \mathbb{L}). \]
Here we have used that the sections $s_i$ are linearly independent. The domain of this embedding can also be presented as the quotient of $C^{n}$ by an $n+1$-dimensional $\mathbb{Q}_p$-vector subspace --- indeed, taking any codimension 1 subspace $C^n \subseteq C^{n+1}$, we have $B^{+,\varphi^{n+1}=p}_\crys + C^n=C^{n+1}$ (e.g. by taking cohomology of the associated modification sequence on the Fargues-Fontaine curve, or by realizing the map from $B^{+,\varphi^{n+1}=p}_\crys$ onto the quotient $C^{n+1}/C^n$ as the surjective logarithm map for a height $n+1$ one-dimensional formal group as in \cite{ScholzeWeinstein.ModuliOfpDivisibleGroups}), so that \[ C^{n+1}/B^{+,\varphi^{n+1}=p}_\crys = C^{n}/B^{+,\varphi^{n+1}=p}_\crys \cap C^n \]
where the intersection on the right is an $n+1$ dimensional $\mathbb{Q}_p$-vector space (e.g. by the taking cohomology of the same modification sequence or by realizing it as the rational Tate module of the same height $n+1$ one-dimensional formal group). In particular, this space has the same cardinality as $C$, so we conclude. 

\begin{remark}
It is known, e.g., for affinoid balls that the $\mathbb{F}_p$-cohomology is infinite dimensional in degree $1$ and depends on $C$. This can be used to show bigness of the fundamental group as in the second construction in the proof of \cref{theorem.generators}. However, for a projective variety, the $\mathbb{F}_p$-cohomology is finite dimensional (it is the same as the algebraic \'{e}tale cohomology) and the profinite \'{e}tale fundamental group is the same as the algebraic one thus stable under change of $C$. Thus one needs $\mathbb{Q}_p$-local systems to get a result for projective varieties (including abelian varieties).
\end{remark}

\end{proof}

\bibliographystyle{plain}
\bibliography{references, preprints}

\end{document}